\documentclass[11pt]{article}
\usepackage{amsmath,amsfonts,amssymb, amsthm,enumerate,graphicx}
\usepackage{tikz,authblk}
\usepackage{subfig}
\usepackage{url}
\newtheorem{theorem} {{\textsf{Theorem}}}
\newtheorem{proposition}[theorem]{{\textsf{Proposition}}}
\newtheorem{corollary}[theorem]{{\textsf{Corollary}}}

\newtheorem{definition}[theorem]{{\textsf{Definition}}}
\newtheorem{remark}[theorem]{{\textsf{Remark}}}
\newtheorem{example}[theorem]{{\textsf{Example}}}
\newtheorem{lemma}[theorem]{{\textsf{Lemma}}}

\begin{document}
\title{3-regular colored graphs and classification of surfaces}
\author{Biplab Basak}
\date{}
\maketitle
\vspace{-10mm}
\begin{center}

\noindent {\small Theoretical Statistics and Mathematics Unit, Indian Statistical Institute, Bangalore 560\,059, India.}

\noindent {\small {\em E-mail address:} \url{biplab8654@gmail.com}}

\medskip

\date{January 17, 2016}
\end{center}
\hrule

\begin{abstract}
Motivated by the theory of crystallizations, we consider an equivalence relation on the class of $3$-regular colored graphs and prove that up to this equivalence  (a) there exists a unique contracted 3-regular colored graph if the number of vertices is $4m$ and  (b) there are exactly two such graphs if the number of vertices is $4m+2$ for each $m\geq 1$. Using this, we present a simple proof of the classification of closed surfaces.
\end{abstract}

\noindent {\small {\em MSC 2010\,:} Primary 05C15. Secondary 05C10; 57Q15; 57N05; 57Q05.

}

\medskip

\hrule

\section{Introduction and Results}\label{sec:intro}
A {\em $3$-regular colored graph} is a loopless 3-regular multigraph which has a proper edge-coloring with three colors. Such a graph is called {\em contracted} if each bi-colored cycle is Hamiltonian. For two 3-regular colored graphs $G$ and $H$ with same color set, the  connected sum is the graph obtained from $G$ and $H$ by deleting any two vertices $u\in G$ and $v \in H$ and welding the ``hanging'' edges of the same colors. Two 3-regular colored graphs will be called {\em $\mathcal{D}$-equivalent} if one can be obtained from the other by a finite sequence of  moves where a move is either a simple cut-and-glue move or interchanging the vertices of connected sums (cf. Definition \ref{Def:move} for details).

Existence of bi-colored Hamiltonian cycles shows that a contracted 3-regular colored graph has even number of vertices. There exists a unique contracted 3-regular colored graph with two vertices. The complete graph $K_4$ together with a proper edge-coloring with three colors is the unique contracted 3-regular colored graph (say, $\mathcal{P}_1$) with four vertices. The complete bipartite graph $K_{3,3}$ together with a proper edge-coloring with three colors is the unique contracted  bipartite 3-regular colored graph (say, $\mathcal{T}_1$) with six vertices (cf. Example \ref{eg:6-vertex}).

For $m\geq 1$, let $\mathcal{P}_m$ (resp., $\mathcal{T}_{m}$) denote  a connected sum of $m$ copies of $\mathcal{P}_1$ (resp., $\mathcal{T}_1$). Then $\mathcal{P}_m$ is a $(2m+2)$-vertex contracted 3-regular  colored graph and $\mathcal{T}_{m}$ is a $(4m+2)$-vertex contracted 3-regular colored graph. Here we prove

\begin{theorem} \label{theorem:graph}
For $n\geq 4$, let $G$ be a contracted $3$-regular colored graph with $n$ vertices. If $n=4m$ for some $m$ then $G$ is $\mathcal{D}$-equivalent to $\mathcal{P}_{2m-1}$. If $n=4m+2$ for some $m$ then $G$ is $\mathcal{D}$-equivalent to $\mathcal{P}_{2m}$ or $\mathcal{T}_{m}$.
\end{theorem}

As an application of the above theorem, we present a simple proof of the following classification of closed connected surfaces (cf. \cite{fw99}).

\begin{corollary}\label{corollary:classification}
If $S$ is a closed connected surface then $S$ is homeomorphic to $\mathbb{S}^2$, $\#_g (\mathbb{S}^1 \times \mathbb{S}^1)$ for some $g\geq 1$ or $\#_h \mathbb{RP}^2$ for some $h \geq 1$.
\end{corollary}
\section{Preliminaries}  \label{sec:prelim}

\subsection{Colored Graphs}\label{prem:graph}

A graph $\Gamma$ is called {\em $3$-regular} if the number of edges adjacent to each vertex is $3$.  A {\it 3-regular colored graph} is a pair $(\Gamma,\gamma)$, where $\Gamma= (V(\Gamma),$ $E(\Gamma))$ is a loopless 3-regular multigraph and $\gamma : E(\Gamma) \to \{0,1, 2\}$ is a proper edge-coloring (i.e., $\gamma(e) \ne \gamma(f)$ for any pair $e,f$ of adjacent edges). We shall  simply say $\Gamma$ is a 3-regular colored graph when the coloring is understood. Two $3$-regular colored graphs $G, H$ are called {\em isomorphic} and is denoted by $G \cong H$ if one can be obtained from the other by renaming the vertices so that equally colored edges in $G$ correspond to equally colored edges in $H$. We identify two colored graphs if they are isomorphic. 
A graph is called {\em bipartite} if the vertices of the graph have a partition with two parts such that end points of every edge lie in different parts of the partition.
We refer to \cite{bm08} for standard terminology on graphs.

Let $(\Gamma,\gamma)$ be a 3-regular colored graph with color set $\{0, 1, 2\}$. For any pair of colors $i\neq j \in \{0,1, 2\}$, the
graph $\Gamma_{\{i,j\}} =(V(\Gamma), \gamma^{-1}(\{i,j\}))$ is a $2$-regular colored graph with edge-coloring $\gamma|_{\gamma^{-1}(\{i,j\})}$.
If $\Gamma_{\{i,j\}}$ is connected (equivalently, if $\Gamma_{\{i,j\}}$ is a Hamiltonian cycle) for all $i \neq j \in \{0,1, 2\}$, then  $(\Gamma,\gamma)$ is called {\em contracted}. Clearly, if $(\Gamma,\gamma)$ is contracted and have more than two vertices then $\Gamma$ is simple.

Let $\Gamma^1 = (V_1, E_1)$ and $\Gamma^2=(V_2, E_2)$ be two $3$-regular colored graphs with same colors $0,1,2$. Taking isomorphic copies (if necessary), we can assume that  $\Gamma^1$ and $\Gamma^2$ are disjoint. For $1\leq i\leq 2$, let $v_i \in V_i$. Consider the graph $\Gamma$ which is
obtained from $(\Gamma^1 - v_1) \sqcup (\Gamma^2  - v_2)$ by adding $3$ new edges
$u_{j,1}u_{j,2}$ of color $j$ for $0\leq j\leq 2$,
where $v_iu_{j,i}$ are edges of color $j$ in $\Gamma^i$ for $0\leq j\leq 2$, $1\leq i\leq 2$. (Here $\Gamma^i - v_i$ is the {\em vertex-deleted subgraph} $\Gamma^i[V_i\setminus\{v_i\}]$.) The colored graph $\Gamma$ is
called a {\em connected sum} of $\Gamma^1$ and $\Gamma^2$, and is denoted by $\Gamma^1\#_{v_1v_2}\Gamma^2$.

For us, a bipartite graph has a fixed bi-partition together with one part with black vertices (which are presented by  black dots `$\bullet$') and the other part with white vertices (which are presented by  white dots `$\circ$'). Let $\Gamma^1$, $\Gamma^2$ be two graphs and $v_1 \in V(\Gamma^1)$, $v_2\in V(\Gamma^2)$.
If $v_1$ and $v_2$ both are either black vertices or white vertices then we say they are of {\em same type}. Otherwise, we say  $v_1$ and $v_2$  are of {\em different types}. If at least one of $\Gamma^1$, $\Gamma^2$ is non-bipartite then, by $\Gamma^1\#\Gamma^2$ we denote a connected sum  $\Gamma^1\#_{v_1v_2}\Gamma^2$ for some arbitrary chosen vertices $v_1,v_2$. If  $\Gamma^1$, $\Gamma^2$ both are bipartite then, by $\Gamma^1\#\Gamma^2$ we denote a connected sum  $\Gamma^1\#_{v_1v_2}\Gamma^2$ where $v_1,v_2$ are of different types.

\begin{definition}
{\rm For two colored graphs $\Gamma^1$ and $\Gamma^2$, let $\Gamma=\Gamma^1\#_{uv}\Gamma^2$ and $\Gamma'=\Gamma^1\#_{u'v'}\Gamma^2$, where $u,v$ (resp., $u',v'$) are of different types if  $\Gamma^1,\Gamma^2$ both are bipartite. If $(u, v)\neq (u',v')$ then we say  $\Gamma$ (resp., $\Gamma'$) is obtained from $\Gamma'$ (resp., $\Gamma$) by {\em interchanging the vertices of connected sums}.}
\end{definition}

Let $\Gamma$ be a 3-regular colored graph (with color set $\{0,1,2\}$). Let $G$ and $H$ be two disjoint 2-regular colored graphs with colors $0$ and $1$ such that  (i)  $V(G) \cup V(H)=V(\Gamma)\cup\{z_1,z_2\}$, (ii)  $z_1 \in V(G)$, $z_2\in V(H)$, (iii) $\Gamma_{\{0,1\}}=G \#_{z_1z_2}H$. Now we construct a 3-regular colored graph $\bar \Gamma$ such that (a) $\bar {\Gamma}_{\{0,1\}}=G \sqcup H$, (b) $z_1$ and $z_2$ are joined by an edge of color 2, and (c) two vertices $u,v \not \in \{ z_1,z_2\}$ of $\bar \Gamma$ are joined by an edge of color 2 if and only if $u$ and $v$ are joined by an edge of color 2 in $\Gamma$. The process from $\Gamma$ to $\bar \Gamma$ is called  a {\em simple cut} with the chosen color $2$ and the chosen pair $(z_1,z_2)$. The reverse process is called a {\em simple glueing}.

\begin{definition}\label{def:cg}
{\rm Let $\Gamma$ and $\Gamma'$ be two 3-regular colored graphs (with color set $\{0,1,2\}$). We say that $\Gamma'$ is obtained from $\Gamma$ by {\em a
simple cut-and-glue move} if there exists a 3-regular colored graph $\bar \Gamma$ such that $\bar \Gamma$ is obtained from $\Gamma$ by a simple cut with the chosen color $2$ and a chosen pair $(z_1,z_2)$ and $\Gamma'$ is obtained from $\bar \Gamma$ by a simple glueing (equivalently, $\bar \Gamma$ is obtained from $\Gamma'$ by a simple cut) with the chosen color $2$ and a chosen pair $(w_1,w_2)$.
}
\end{definition}

\begin{definition}\label{Def:move}
{\rm Two 3-regular colored graphs $\Gamma$ and $\Gamma'$ are called {\em $\mathcal{D}$-equivalent} (and is denoted by $\Gamma \approx_{\mathcal{D}}\Gamma'$) if
there exists a finite sequence \{$\Gamma^i\}_{i=1}^{n}$ of  3-regular  colored graphs such that $\Gamma^1=\Gamma$ and $\Gamma^n=\Gamma'$, and
$\Gamma^i$ is obtained from $ \Gamma^{i-1}$ by a simple cut-and-glue move or by interchanging the vertices of connected sums, for $2 \leq i \leq n$.}
\end{definition}

\begin{example}[Contracted 3-regular colored graphs up to 6 vertices]\label{eg:6-vertex}
{\rm The graph $\mathcal{L}$ given in Figure \ref{fig:graph} is a 3-regular colored graph with two vertices. Clearly, $\mathcal{L}$ is the unique 3-regular colored graph with two vertices. Note that $\mathcal{L}$ is contracted.
The graph $\mathcal{P}_1$ given in Figure \ref{fig:graph} is a  contracted 3-regular colored graph with four vertices. It is a non-bipartite graph. Clearly, $\mathcal{P}_1$ is the unique contracted 3-regular colored graph with four vertices.
The graphs $\mathcal{T}_1$ and $\mathcal{P}_2$ given in Figure \ref{fig:graph} are contracted 3-regular colored graphs with six vertices. Note that $\mathcal{T}_1$ is a bipartite graph and  $\mathcal{P}_2$ is a non-bipartite graph. It is not difficult to see that any 6-vertex contracted 3-regular colored graph is isomorphic to $\mathcal{T}_1$ or $\mathcal{P}_2$.
\begin{figure}[ht]
\tikzstyle{ver}=[]
\tikzstyle{vert}=[circle, draw, fill=black!100, inner sep=0pt, minimum width=4pt]
\tikzstyle{vertex}=[circle, draw, fill=black!00, inner sep=0pt, minimum width=4pt]
\tikzstyle{edge} = [draw,thick,-]
\centering

\begin{tikzpicture}[scale=0.32]
\begin{scope}[shift={(-16,0)}]
\node[vert] (v1) at (-7,0){};
\node[vertex] (v2) at (-1,0){};
\node[ver] () at (-8,0){$v_{1}$};
\node[ver] () at (0,0){$v_{2}$};

\draw[edge] plot [smooth,tension=1] coordinates{(v1)(-4,1)(v2)};
\draw[line width=3pt, line cap=round, dash pattern=on 0pt off 2\pgflinewidth] plot [smooth,tension=1] coordinates{(v1)(-4,1)(v2)};
\path[edge] (v1) -- (v2);
\draw[edge, dashed] plot [smooth,tension=1] coordinates{(v1)(-4,-1)(v2)};
\node[ver] () at (-4,-3.5){$\mathcal{L}$};
\end{scope}

\begin{scope}[shift={(-10,0)}]
\foreach \x/\y in {45/v_{2},225/v_{4}}{
\node[ver] (\y) at (\x:4){$\y$};
    \node[vertex] (\y) at (\x:3){};
}

\foreach \x/\y in {135/v_{3},315/v_{1}}{
\node[ver] (\y) at (\x:4){$\y$};
    \node[vert] (\y) at (\x:3){};
}
\foreach \x/\y in {v_{1}/v_{2},v_{2}/v_{3},v_{3}/v_{4},v_{4}/v_{1}}{
\path[edge] (\x) -- (\y);}
\foreach \x/\y in {v_{1}/v_{2},v_{3}/v_{4}}{
\draw [line width=3pt, line cap=round, dash pattern=on 0pt off 2\pgflinewidth]  (\x) -- (\y);
}
\foreach \x/\y in {v_{1}/v_{3},v_{2}/v_{4}}{
\path[edge, dashed] (\x) -- (\y);}
\node[ver] () at (0,-3.5){$\mathcal{P}_1$};
\end{scope}

\begin{scope}[shift={(0,0)}]
\foreach \x/\y in {300/v_{2},180/v_{4},60/v_{6}}{
\node[ver] (\y) at (\x:4){$\y$};
    \node[vertex] (\y) at (\x:3){};
}
\foreach \x/\y in {240/v_{3},120/v_{5},0/v_{1}}{
\node[ver] (\y) at (\x:4){$\y$};
    \node[vert] (\y) at (\x:3){};
}
\foreach \x/\y in {v_{1}/v_{2},v_{2}/v_{3},v_{3}/v_{4},v_{4}/v_{5},v_{5}/v_{6},v_{6}/v_{1}}{
\path[edge] (\x) -- (\y);}
\foreach \x/\y in {v_{1}/v_{2},v_{3}/v_{4},v_{5}/v_{6}}{
\draw [line width=3pt, line cap=round, dash pattern=on 0pt off 2\pgflinewidth]  (\x) -- (\y);
}
\foreach \x/\y in {v_{1}/v_{4},v_{2}/v_{5},v_{3}/v_{6}}{
\path[edge, dashed] (\x) -- (\y);}
\node[ver] () at (0,-3.5){$\mathcal{T}_1$};
\end{scope}

\begin{scope}[shift={(11,0)}]
\foreach \x/\y in {300/v_{2},180/v_{4},60/v_{6}}{
\node[ver] (\y) at (\x:4){$\y$};
    \node[vertex] (\y) at (\x:3){};
}
\foreach \x/\y in {240/v_{3},120/v_{5},0/v_{1}}{
\node[ver] (\y) at (\x:4){$\y$};
    \node[vert] (\y) at (\x:3){};
}
\foreach \x/\y in {v_{1}/v_{2},v_{2}/v_{3},v_{3}/v_{4},v_{4}/v_{5},v_{5}/v_{6},v_{6}/v_{1}}{
\path[edge] (\x) -- (\y);}
\foreach \x/\y in {v_{1}/v_{2},v_{3}/v_{4},v_{5}/v_{6}}{
\draw [line width=3pt, line cap=round, dash pattern=on 0pt off 2\pgflinewidth]  (\x) -- (\y);
}
\foreach \x/\y in {v_{1}/v_{4},v_{2}/v_{6},v_{3}/v_{5}}{
\path[edge, dashed] (\x) -- (\y);}
\node[ver] () at (0,-3.5){$\mathcal{P}_2$};
\end{scope}

\begin{scope}[shift={(20,-3)}]
\node[ver] (308) at (-1,4){$0$};
\node[ver] (300) at (-1,3){$1$};
\node[ver] (301) at (-1,2){$2$};
\node[ver] (309) at (2,4){};
\node[ver] (304) at (2,3){};
\node[ver] (305) at (2,2){};

\path[edge] (300) -- (304);
\path[edge] (308) -- (309);
\draw [line width=3pt, line cap=round, dash pattern=on 0pt off 2\pgflinewidth]  (308) -- (309);
\path[edge, dashed] (301) -- (305);
\end{scope}
\end{tikzpicture}
\caption{Contracted 3-regular colored graphs up to 6 vertices.}\label{fig:graph}
\end{figure}
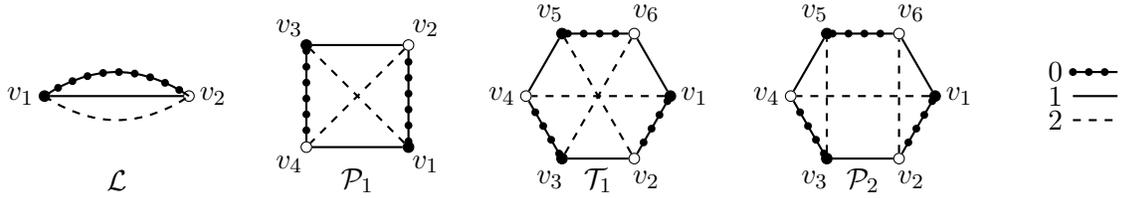

For $m\geq 1$, $\mathcal{P}_m$ (resp., $\mathcal{T}_{m}$) is a connected sum of $m$ copies of $\mathcal{P}_1$ (resp., $\mathcal{T}_1$).}
\end{example}

\subsection{Crystallizations} \label{crystal}
Each $3$-regular colored graph $(\Gamma,\gamma)$ determines a  $2$-dimensional simplicial cell-complex ${\mathcal K}(\Gamma)$ as follows:
\begin{itemize}
\item{} for every vertex $v\in V(\Gamma)$, take a $2$-simplex $\sigma(v)$ and label its three vertices by the colors $0,1 $ and $2$ respectively,
\item{} for every edge (of $\Gamma$) of color $i$ between $v,w\in V(\Gamma)$, identify the edges of $\sigma(v)$ and $\sigma(w)$ opposite to $i$-labelled vertices, so that equally labelled vertices coincide.
\end{itemize}

If $\Gamma$ is contracted and the underlying topological space of ${\mathcal K}(\Gamma)$ is homeomorphic to a closed surface $M$, then the $3$-regular colored graph $(\Gamma,\gamma)$ is called a {\it crystallization} of $M$. For more on simplicial cell complexes and  related notions see \cite{bj84}. From the works on crystallizations, we know the following results for surfaces.

\begin{proposition}[\cite{pe74}, Proposition 4.1]\label{prop:Pezzana}
Every closed connected surface admits a crystallization.
\end{proposition}

\begin{proposition}[\cite{cgp80}, Proposition 15]\label{prop:cipartite}
Let $(\Gamma,\gamma)$ be a crystallization of a closed connected surface $M$. Then $M$ is orientable if and only if $\Gamma$ is bipartite.
\end{proposition}

\begin{proposition}[\cite{fgg86}, Theorem 5]\label{prop:preliminaries}
For $1\leq i \leq 2$, let $(\Gamma^i,\gamma^i)$ be a crystallization of a closed connected surface $M_i$ (with the induced orientation if $M_i$ is orientable). Then any connected sum $\Gamma^1 \# \Gamma^2$ is a crystallization of $M_1\#M_2$.
\end{proposition}

\begin{proposition}[\cite{fg82}, the main result] \label{prop:cut-glue}
Let $(\Gamma, \gamma)$ and $(\Gamma',\gamma')$ be crystallizations of closed connected surfaces $S_1$ and $S_2$ respectively. {\rm (a)} If $(\Gamma,\gamma)$ can be obtained from $(\Gamma',\gamma')$ by a simple cut-and-glue move then $S_1$ and $S_2$ are homeomorphic. {\rm (b)} Conversely, if $S_1$ and $S_2$ are homeomorphic then $(\Gamma', \gamma')$ can be obtained from $(\Gamma, \gamma)$ by a finite sequence of simple cut-and-glue moves.
\end{proposition}

Propositions \ref{prop:Pezzana}, \ref{prop:cipartite}, \ref{prop:preliminaries} and   \ref{prop:cut-glue} (a) are known for manifolds of dimension $d \geq 2$.
A different version of Proposition \ref{prop:cut-glue} (b) is also known for  dimensions $d\geq 2$.
From Propositions \ref{prop:preliminaries}, \ref{prop:cut-glue} and Definition \ref{Def:move}, we get the following.

\begin{corollary} \label{cor:homeomorphic}
Let $(\Gamma,\gamma)$ and $(\Gamma',\gamma')$ be crystallizations of closed connected surfaces $S_1$ and $S_2$ respectively. Then $(\Gamma,\gamma)$ and $(\Gamma',\gamma')$ are $\mathcal{D}$-equivalent if and only if $S_1$ and $S_2$ are homeomorphic.
\end{corollary}

\section{Proof of Theorem \ref{theorem:graph}}\label{subsec:proof1}
In this section, we prove Theorem \ref{theorem:graph} using the following four lemmas.

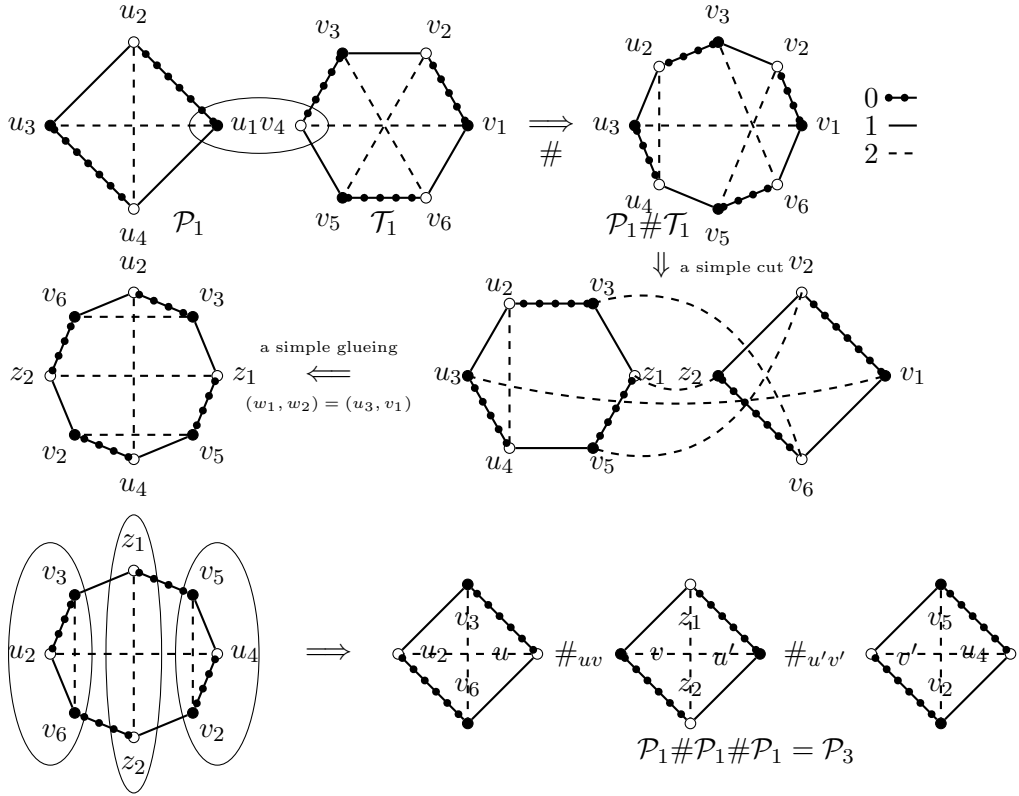
\begin{figure}[ht]
\tikzstyle{ver}=[]
\tikzstyle{vert}=[circle, draw, fill=black!100, inner sep=0pt, minimum width=4pt]
\tikzstyle{vertex}=[circle, draw, fill=black!00, inner sep=0pt, minimum width=4pt]
\tikzstyle{edge} = [draw,thick,-]
\centering

\begin{tikzpicture}[scale=0.37]
\begin{scope}[shift={(-12,0)}]
\foreach \x/\y in {90/u_{2},270/u_{4}}{
\node[ver] (\y) at (\x:4){$\y$};
    \node[vertex] (\y) at (\x:3){};
}

\foreach \x/\y in {180/u_{3},0/u_{1}}{
\node[ver] (\y) at (\x:4){$\y$};
    \node[vert] (\y) at (\x:3){};
}
\foreach \x/\y in {u_{1}/u_{2},u_{2}/u_{3},u_{3}/u_{4},u_{4}/u_{1}}{
\path[edge] (\x) -- (\y);}
\foreach \x/\y in {u_{1}/u_{2},u_{3}/u_{4}}{
\draw [line width=3pt, line cap=round, dash pattern=on 0pt off 2\pgflinewidth]  (\x) -- (\y);
}
\foreach \x/\y in {u_{1}/u_{3},u_{2}/u_{4}}{
\path[edge, dashed] (\x) -- (\y);}
\node[ver] () at (2,-3.5){$\mathcal{P}_1$};
\draw (4.5,0) ellipse (2.5cm and 1cm);
\end{scope}

\begin{scope}[shift={(-3,0)}]
\foreach \x/\y in {300/v_{6},180/v_{4},60/v_{2}}{
\node[ver] (\y) at (\x:4){$\y$};
    \node[vertex] (\y) at (\x:3){};
}
\foreach \x/\y in {240/v_{5},120/v_{3},0/v_{1}}{
\node[ver] (\y) at (\x:4){$\y$};
    \node[vert] (\y) at (\x:3){};
}
\foreach \x/\y in {v_{1}/v_{2},v_{2}/v_{3},v_{3}/v_{4},v_{4}/v_{5},v_{5}/v_{6},v_{6}/v_{1}}{
\path[edge] (\x) -- (\y);}
\foreach \x/\y in {v_{1}/v_{2},v_{3}/v_{4},v_{5}/v_{6}}{
\draw [line width=3pt, line cap=round, dash pattern=on 0pt off 2\pgflinewidth]  (\x) -- (\y);
}
\foreach \x/\y in {v_{1}/v_{4},v_{2}/v_{5},v_{3}/v_{6}}{
\path[edge, dashed] (\x) -- (\y);}
\node[ver] () at (0,-3.5){$\mathcal{T}_1$};
\end{scope}

\begin{scope}[shift={(9,0)}]
\foreach \x/\y in {315/v_{6},225/u_{4},135/u_{2},45/v_{2}}{
\node[ver] (\y) at (\x:4){$\y$};
    \node[vertex] (\y) at (\x:3){};
}
\foreach \x/\y in {270/v_{5},180/u_{3},90/v_{3},0/v_{1}}{
\node[ver] (\y) at (\x:4){$\y$};
    \node[vert] (\y) at (\x:3){};
}
\foreach \x/\y in {v_{1}/v_{2},v_{2}/v_{3},v_{5}/v_{6},v_{6}/v_{1},v_{3}/u_{2},u_{2}/u_{3},u_{3}/u_{4},u_{4}/v_{5}}{
\path[edge] (\x) -- (\y);}
\foreach \x/\y in {v_{1}/v_{2},v_{5}/v_{6},v_{3}/u_{2},u_{3}/u_{4}}{
\draw [line width=3pt, line cap=round, dash pattern=on 0pt off 2\pgflinewidth]  (\x) -- (\y);}

\foreach \x/\y in {v_{1}/u_{3},v_{2}/v_{5},v_{3}/v_{6},u_{4}/u_{2}}{
\path[edge, dashed] (\x) -- (\y);}
\node[ver] () at (-2.5,-3.6){$\mathcal{P}_1\#\mathcal{T}_1$};
\node[ver] () at (-6,0){$ \Longrightarrow $};
\node[ver] () at (-6,-1){$ \# $};
\end{scope}

\begin{scope}[shift={(15.5,-3)}]
\node[ver] (308) at (-1,4){$0$};
\node[ver] (300) at (-1,3){$1$};
\node[ver] (301) at (-1,2){$2$};
\node[ver] (309) at (1,4){};
\node[ver] (304) at (1,3){};
\node[ver] (305) at (1,2){};

\path[edge] (300) -- (304);
\path[edge] (308) -- (309);
\draw [line width=3pt, line cap=round, dash pattern=on 0pt off 2\pgflinewidth]  (308) -- (309);
\path[edge, dashed] (301) -- (305);
\end{scope}

\begin{scope}[shift={(12,-9)}]
\foreach \x/\y in {90/v_{2},270/v_{6}}{
\node[ver] (\y) at (\x:4){$\y$};
    \node[vertex] (\y) at (\x:3){};
}

\foreach \x/\y in {180/z_{2},0/v_{1}}{
\node[ver] (\y) at (\x:4){$\y$};
    \node[vert] (\y) at (\x:3){};
}
\foreach \x/\y in {v_{1}/v_{2},v_{2}/z_{2},z_{2}/v_{6},v_{6}/v_{1}}{
\path[edge] (\x) -- (\y);}
\foreach \x/\y in {v_{1}/v_{2},z_{2}/v_{6}}{
\draw [line width=3pt, line cap=round, dash pattern=on 0pt off 2\pgflinewidth]  (\x) -- (\y);
}
\node[ver] () at (-3,4){$\Downarrow$ \tiny{a simple cut}};
\end{scope}

\begin{scope}[shift={(3,-9)}]
\foreach \x/\y in {300/v_{5},180/u_{3},60/v_{3}}{
\node[ver] (\y) at (\x:3.7){$\y$};
    \node[vert] (\y) at (\x:3){};
}
\foreach \x/\y in {240/u_{4},120/u_{2},0/z_{1}}{
\node[ver] (\y) at (\x:3.7){$\y$};
    \node[vertex] (\y) at (\x:3){};
}
\foreach \x/\y in {z_{1}/v_{3},v_{3}/u_{2},u_{2}/u_{3},u_{3}/u_{4},u_{4}/v_{5},v_{5}/z_{1}}{
\path[edge] (\x) -- (\y);}
\foreach \x/\y in {z_{1}/v_{5},u_{4}/u_{3},v_{3}/u_{2}}{
\draw [line width=3pt, line cap=round, dash pattern=on 0pt off 2\pgflinewidth]  (\x) -- (\y);
}
\foreach \x/\y in {u_{2}/u_{4}}{
\path[edge, dashed] (\x) -- (\y);}
\draw[edge, dashed] plot [smooth,tension=1] coordinates{(z_{1})(4.5,-.5)(z_{2})};
\draw[edge, dashed] plot [smooth,tension=1] coordinates{(u_{3})(4.5,-1)(v_{1})};
\draw[edge, dashed] plot [smooth,tension=1] coordinates{(v_{3})(6,2)(v_{6})};
\draw[edge, dashed] plot [smooth,tension=1] coordinates{(v_{2})(6,-2)(v_{5})};
\end{scope}

\begin{scope}[shift={(-12,-9)}]
\foreach \x/\y in {315/v_{5},225/v_{2},135/v_{6},45/v_{3}}{
\node[ver] (\y) at (\x:4){$\y$};
    \node[vert] (\y) at (\x:3){};
}
\foreach \x/\y in {270/u_{4},180/z_{2},90/u_{2},0/z_{1}}{
\node[ver] (\y) at (\x:4){$\y$};
    \node[vertex] (\y) at (\x:3){};
}
\foreach \x/\y in {z_{1}/v_{3},v_{3}/u_{2},u_{4}/v_{5},v_{5}/z_{1},u_{2}/v_{6},v_{6}/z_{2},z_{2}/v_{2},v_{2}/u_{4}}{
\path[edge] (\x) -- (\y);}
\foreach \x/\y in {z_{1}/v_{5},u_{2}/v_{3},u_{4}/v_{2},z_{2}/v_{6}}{
\draw [line width=3pt, line cap=round, dash pattern=on 0pt off 2\pgflinewidth]  (\x) -- (\y);}

\foreach \x/\y in {z_{1}/z_{2},u_{2}/u_{4},v_{3}/v_{6},v_{5}/v_{2}}{
\path[edge, dashed] (\x) -- (\y);}
\node[ver] () at (7,1){\tiny{a simple glueing}};
\node[ver] () at (7,0){$ \Longleftarrow $};
\node[ver] () at (7,-1){\tiny{$(w_1,w_2)=(u_3,v_1)$}};
\end{scope}
\begin{scope}[shift={(-12,-19)}, rotate=90]
\foreach \x/\y in {315/v_{5},225/v_{2},135/v_{6},45/v_{3}}{
\node[ver] (\y) at (\x:4){$\y$};
    \node[vert] (\y) at (\x:3){};
}
\foreach \x/\y in {270/u_{4},180/z_{2},90/u_{2},0/z_{1}}{
\node[ver] (\y) at (\x:4){$\y$};
    \node[vertex] (\y) at (\x:3){};
}
\foreach \x/\y in {z_{1}/v_{3},v_{3}/u_{2},u_{4}/v_{5},v_{5}/z_{1},u_{2}/v_{6},v_{6}/z_{2},z_{2}/v_{2},v_{2}/u_{4}}{
\path[edge] (\x) -- (\y);}
\foreach \x/\y in {z_{1}/v_{5},u_{2}/v_{3},u_{4}/v_{2},z_{2}/v_{6}}{
\draw [line width=3pt, line cap=round, dash pattern=on 0pt off 2\pgflinewidth]  (\x) -- (\y);}

\foreach \x/\y in {z_{1}/z_{2},u_{2}/u_{4},v_{3}/v_{6},v_{5}/v_{2}}{
\path[edge, dashed] (\x) -- (\y);}
\draw (0,0) ellipse (5cm and 1cm);
\draw (0,-3) ellipse (4cm and 1.5cm);
\draw (0,3) ellipse (4cm and 1.5cm);
\node[ver] () at (0,-7){$ \Longrightarrow $};
\end{scope}

\begin{scope}[shift={(0,-19)}]
\foreach \x/\y in {90/v_{3},270/v_{6}}{
\node[ver] (\y) at (\x:1.2){$\y$};
    \node[vert] (\y) at (\x:2.5){};}

\foreach \x/\y in {180/u_{2},0/u}{
\node[ver] (\y) at (\x:1.2){$\y$};
    \node[vertex] (\y) at (\x:2.5){};}
\foreach \x/\y in {u/v_{3},v_{3}/u_{2},u_{2}/v_{6},v_{6}/u}{
\path[edge] (\x) -- (\y);}
\foreach \x/\y in {u/v_{3},u_{2}/v_{6}}{
\draw [line width=3pt, line cap=round, dash pattern=on 0pt off 2\pgflinewidth]  (\x) -- (\y);}
\foreach \x/\y in {u/u_{2},v_{3}/v_{6}}{
\path[edge, dashed] (\x) -- (\y);}
\node[ver] () at (4,0){$\#_{uv}$};
\end{scope}

\begin{scope}[shift={(8,-19)}]
\foreach \x/\y in {90/z_{1},270/z_{2}}{
\node[ver] (\y) at (\x:1.2){$\y$};
    \node[vertex] (\y) at (\x:2.5){};}

\foreach \x/\y in {180/v,0/u'}{
\node[ver] (\y) at (\x:1.2){$\y$};
    \node[vert] (\y) at (\x:2.5){};}
\foreach \x/\y in {u'/z_{1},z_{1}/v,v/z_{2},z_{2}/u'}{
\path[edge] (\x) -- (\y);}
\foreach \x/\y in {u'/z_{1},v/z_{2}}{
\draw [line width=3pt, line cap=round, dash pattern=on 0pt off 2\pgflinewidth]  (\x) -- (\y);}
\foreach \x/\y in {u'/v,z_{1}/z_{2}}{
\path[edge, dashed] (\x) -- (\y);}
\node[ver] () at (4.5,0){$\#_{u'v'}$};
\end{scope}

\begin{scope}[shift={(17,-19)}]
\foreach \x/\y in {90/v_{5},270/v_{2}}{
\node[ver] (\y) at (\x:1.2){$\y$};
    \node[vert] (\y) at (\x:2.5){};}

\foreach \x/\y in {180/v',0/u_{4}}{
\node[ver] (\y) at (\x:1.2){$\y$};
    \node[vertex] (\y) at (\x:2.5){};}
\foreach \x/\y in {u_{4}/v_{5},v_{5}/v',v'/v_{2},v_{2}/u_{4}}{
\path[edge] (\x) -- (\y);}
\foreach \x/\y in {u_{4}/v_{5},v'/v_{2}}{
\draw [line width=3pt, line cap=round, dash pattern=on 0pt off 2\pgflinewidth]  (\x) -- (\y);}
\foreach \x/\y in {u_{4}/v',v_{5}/v_{2}}{
\path[edge, dashed] (\x) -- (\y);}
\node[ver] () at (-7,-3.5){$\mathcal{P}_1\#\mathcal{P}_1\#\mathcal{P}_1=\mathcal{P}_3$};
\end{scope}
\end{tikzpicture}
\caption{$\mathcal{P}_3$ is obtained from $\mathcal{P}_1\#\mathcal{T}_1$ by a simple cut-and-glue move.}\label{fig:connectedsum}
\end{figure}

\begin{lemma}\label{lemma:connectedsum}
For $m\geq 1$,  $\mathcal{T}_{m}\#\mathcal{P}_1$ is  $\mathcal{D}$-equivalent to $\mathcal{P}_{2m+1}$.
\end{lemma}

\begin{proof}
We know from Figure \ref{fig:connectedsum} that $\mathcal{P}_3$ can be obtained from $\mathcal{T}_1\#\mathcal{P}_1$ by a simple cut-and-glue move. Therefore, by Definition \ref{Def:move}, $\mathcal{T}_1 \#\mathcal{P}_1$ is $\mathcal{D}$-equivalent to $\mathcal{P}_3$. Thus, the statement is true for $m=1$. Let us assume that the statement is true for $m=n-1$, i.e., $\mathcal{T}_{n-1}\#\mathcal{P}_1$ is  $\mathcal{D}$-equivalent to $\mathcal{P}_{2n-1}$. Now we show that the statement is true for $m=n$.

From the constructions of $\mathcal{P}_{m}$ and $\mathcal{T}_{m}$, we have $\mathcal{P}_i\#\mathcal{P}_j=\mathcal{P}_{i+j}$ and $\mathcal{T}_i\#\mathcal{T}_j=\mathcal{T}_{i+j}$. Since $\mathcal{T}_1 \#\mathcal{P}_1$ is $\mathcal{D}$-equivalent to $\mathcal{P}_3$, it follows that $\mathcal{T}_n \#\mathcal{P}_1$ is $\mathcal{D}$-equivalent to $\mathcal{T}_{n-1} \#\mathcal{P}_3$. On the other hand, since $\mathcal{T}_{n-1}\#\mathcal{P}_1$ is  $\mathcal{D}$-equivalent to $\mathcal{P}_{2n-1}$, $\mathcal{T}_{n-1}\#\mathcal{P}_3$ is  $\mathcal{D}$-equivalent to $\mathcal{P}_{2n+1}$.
Therefore,
\begin{align*}
\mathcal{T}_{n}\#\mathcal{P}_1 \approx_{\mathcal{D}} \mathcal{T}_{n-1} \#\mathcal{P}_3  \approx_{\mathcal{D}} \mathcal{P}_{2n+1}.
\end{align*}
Thus the statement is true for $m=n$. Now, the result follows by the induction principle.
\end{proof}

\begin{lemma}\label{lemma:nograph}
For $m \geq 1$, there is no contracted bipartite  $3$-regular colored graph with $4m$ vertices.
\end{lemma}

\begin{proof}
If possible, let $\Gamma$ be a contracted  bipartite 3-regular colored graph with $4m$ vertices. Let $\{u_i\,|\,1\leq i \leq 2m\} \sqcup \{v_i\,|\,1\leq i \leq 2m\}$ be the bipartition of the vertex set of $\Gamma$. Let $u_iv_{\sigma_0(i)}$ be the edges of color 0 and $u_iv_{\sigma_1(i)}$ be the edges of color 1 for $1\leq i \leq 2m$ where $\sigma_0,\sigma_1\in$ Sym$(2m)\, (=$ the symmetric group on a finite set of  $2m$ symbols). Since $\Gamma$ is contracted, $\Gamma_{\{0,1\}}$ is a Hamiltonian cycle, equivalently,  $\sigma_1^{-1}\sigma_0$ is a cyclic permutation of order $2m$. Without loss, we can assume $\sigma_0$ is identity, and hence $\sigma_1$ is a cyclic permutation of order $2m$. Let $u_iv_{\sigma_2(i)}$ be the edges of color 2 (Since $\Gamma$ is simple, each edge is uniquely determined by its ends). Since $\Gamma_{\{0,2\}}$ is a Hamiltonian cycle, $\sigma_2$ is a cyclic permutation of order $2m$. Since $\Gamma_{\{1,2\}}$ is a Hamiltonian cycle, $\sigma_2^{-1}\sigma_1$ is a cyclic permutation of order $2m$.

Now, a cyclic permutation of order $2m$ is an odd permutation, and hence $\sigma_1$ and $\sigma_2$ are odd permutations. Therefore, $\sigma_2^{-1}\sigma_1$ is an even permutation, and hence $\sigma_2^{-1}\sigma_1$ can not be a cyclic permutation of order $2m$. This is a contradiction. This proves the result.
\end{proof}

\begin{lemma}\label{lemma:bipartite}
For $m\geq 1$, if $\Gamma$ is a bipartite contracted $3$-regular colored graph with $4m+2$ vertices then $\Gamma$ is $\mathcal{D}$-equivalent to $\mathcal{T}_{m}$.
\end{lemma}

\begin{proof}
Let $\{0,1, 2\}$ be the color set. If $m=1$, then $\Gamma$ is a bipartite contracted 3-regular colored graph with six vertices. Thus, from Example \ref{eg:6-vertex}, $\Gamma$ is unique and is isomorphic to $\mathcal{T}_1$. Trivially, $\Gamma$ is $\mathcal{D}$-equivalent to $\mathcal{T}_{1}$.

Let us assume that the statement is true for $m=q-1$, i.e., if $\Gamma$ is a bipartite contracted 3-regular colored graph with $4q-2$ vertices then $\Gamma$ is $\mathcal{D}$-equivalent to $\mathcal{T}_{q-1}$. We now show that the statement is true for $m=q$.

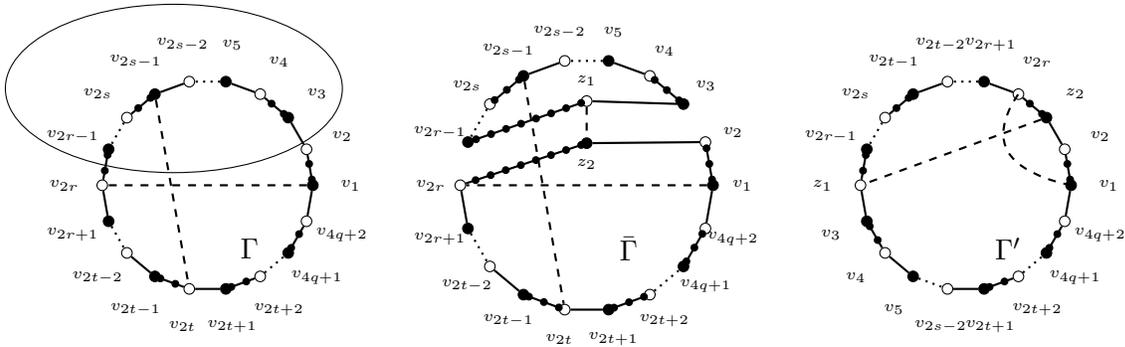
\begin{figure}[ht]
\tikzstyle{ver}=[]
\tikzstyle{vert}=[circle, draw, fill=black!100, inner sep=0pt, minimum width=4pt]
\tikzstyle{vertex}=[circle, draw, fill=black!00, inner sep=0pt, minimum width=4pt]
\tikzstyle{edge} = [draw,thick,-]
\centering
\begin{tikzpicture}[scale=0.56]
\begin{scope}[shift={(0,0)}]
\foreach \x/\y in {20/v_{2},60/v_{4},100/v_{6},140/v_{8},180/v_{10},220/v_{12},260/v_{14},300/v_{16},340/v_{18}}{
    \node[vertex] (\y) at (\x:2.5){};}

\foreach \x/\y in {0/v_{1},40/v_{3},80/v_{5},120/v_{7},160/v_{9},200/v_{11},240/v_{13},280/v_{15},320/v_{17}}{
\node[vert] (\y) at (\x:2.5){};}

\foreach \x/\y in {0/v_{1},20/v_{2},40/v_{3},60/v_{4},80/v_{5},100/v_{2s-2},120/v_{2s-1},140/v_{2s},160/v_{2r-1},180/v_{2r},200/v_{2r+1},220/v_{2t-2},240/v_{2t-1},260/v_{2t},280/v_{2t+1},300/v_{2t+2},320/v_{4q+1},340/v_{4q+2}}{
\node[ver] () at (\x:3.4){\tiny{$\y$}};}

\foreach \x/\y in {v_{1}/v_{2},v_{2}/v_{3},v_{3}/v_{4},v_{4}/v_{5},v_{6}/v_{7},v_{7}/v_{8},v_{9}/v_{10},v_{10}/v_{11},v_{12}/v_{13},v_{13}/v_{14},v_{14}/v_{15},v_{15}/v_{16},v_{17}/v_{18},v_{18}/v_{1}}{\path[edge] (\x) -- (\y);}

\foreach \x/\y in {v_{1}/v_{2},v_{3}/v_{4},v_{7}/v_{8},v_{9}/v_{10},v_{13}/v_{14},v_{15}/v_{16},v_{17}/v_{18}}{
\draw [line width=3pt, line cap=round, dash pattern=on 0pt off 2\pgflinewidth]  (\x) -- (\y);}

\foreach \x/\y in {v_{1}/v_{10},v_{7}/v_{14}}{
\path[edge, dashed] (\x) -- (\y);}
\foreach \x/\y in {v_{5}/v_{6},v_{8}/v_{9},v_{11}/v_{12}, v_{16}/v_{17}}{
\path[edge, dotted] (\x) -- (\y);}
\draw (-.8,2.3) ellipse (4cm and 2cm);
\node[ver] () at (1,-1.5){$\Gamma$};
\end{scope}

\begin{scope}[shift={(9,0)}]
\foreach \x/\y in {20/v_{2},60/v_{4},100/v_{6},140/v_{8},180/v_{10},220/v_{12},260/v_{14},300/v_{16},340/v_{18}}{
    \node[vertex] (\y) at (\x:3){};}
    \node[vertex] (z_1) at (0,2){};
    \node[vert] (z_2) at (0,1){};
    \node[ver] () at (0,2.5){\tiny{$z_1$}};
    \node[ver] () at (0,.5){\tiny{$z_2$}};

\foreach \x/\y in {0/v_{1},40/v_{3},80/v_{5},120/v_{7},160/v_{9},200/v_{11},240/v_{13},280/v_{15},320/v_{17}}{
\node[vert] (\y) at (\x:3){};}

\foreach \x/\y in {0/v_{1},20/v_{2},40/v_{3},60/v_{4},80/v_{5},100/v_{2s-2},120/v_{2s-1},140/v_{2s},160/v_{2r-1},180/v_{2r},200/v_{2r+1},220/v_{2t-2},240/v_{2t-1},260/v_{2t},280/v_{2t+1},300/v_{2t+2},320/v_{4q+1},340/v_{4q+2}}{
\node[ver] () at (\x:3.7){\tiny{$\y$}};}

\foreach \x/\y in {v_{1}/v_{2},v_{3}/v_{4},v_{4}/v_{5},v_{6}/v_{7},v_{7}/v_{8},v_{10}/v_{11},v_{12}/v_{13},v_{13}/v_{14},v_{14}/v_{15},v_{15}/v_{16},v_{17}/v_{18},v_{18}/v_{1},v_{3}/z_1,z_1/v_{9},v_{2}/z_2,z_2/v_{10}}{\path[edge] (\x) -- (\y);}

\foreach \x/\y in {v_{1}/v_{2},v_{3}/v_{4},v_{7}/v_{8},v_{13}/v_{14},v_{15}/v_{16},v_{17}/v_{18},z_1/v_{9},z_2/v_{10}}{
\draw [line width=3pt, line cap=round, dash pattern=on 0pt off 2\pgflinewidth]  (\x) -- (\y);}

\foreach \x/\y in {v_{1}/v_{10},v_{7}/v_{14},z_1/z_2}{
\path[edge, dashed] (\x) -- (\y);}
\foreach \x/\y in {v_{5}/v_{6},v_{8}/v_{9},v_{11}/v_{12}, v_{16}/v_{17}}{
\path[edge, dotted] (\x) -- (\y);}
\node[ver] () at (1,-1.5){$\bar\Gamma$};
\end{scope}

\begin{scope}[shift={(18,0)}]
\foreach \x/\y in {20/v_{2},60/v_{4},100/v_{6},140/v_{8},180/v_{10},220/v_{12},260/v_{14},300/v_{16},340/v_{18}}{
    \node[vertex] (\y) at (\x:2.5){};}

\foreach \x/\y in {0/v_{1},40/v_{3},80/v_{5},120/v_{7},160/v_{9},200/v_{11},240/v_{13},280/v_{15},320/v_{17}}{
\node[vert] (\y) at (\x:2.5){};}

\foreach \x/\y in {0/v_{1},20/v_{2},40/z_2,60/v_{2r},80/v_{2r+1},100/v_{2t-2},120/v_{2t-1},140/v_{2s},160/v_{2r-1},180/z_1,200/v_{3},220/v_{4},240/v_{5},260/v_{2s-2},280/v_{2t+1},300/v_{2t+2},320/v_{4q+1},340/v_{4q+2}}{
\node[ver] () at (\x:3.4){\tiny{$\y$}};}

\foreach \x/\y in {v_{1}/v_{2},v_{2}/v_{3},v_{3}/v_{4},v_{4}/v_{5},v_{6}/v_{7},v_{7}/v_{8},v_{9}/v_{10},v_{10}/v_{11},v_{11}/v_{12},v_{12}/v_{13},v_{14}/v_{15},v_{15}/v_{16},v_{17}/v_{18},v_{18}/v_{1}}{\path[edge] (\x) -- (\y);}

\foreach \x/\y in {v_{1}/v_{2},v_{3}/v_{4},v_{7}/v_{8},v_{9}/v_{10},v_{11}/v_{12},v_{15}/v_{16},v_{17}/v_{18}}{
\draw [line width=3pt, line cap=round, dash pattern=on 0pt off 2\pgflinewidth]  (\x) -- (\y);}

\draw[edge, dashed] plot [smooth,tension=1] coordinates{(v_{1})(1,.8)(v_{4})};
\path[edge, dashed] (v_{3}) -- (v_{10});
\foreach \x/\y in {v_{5}/v_{6},v_{8}/v_{9},v_{13}/v_{14},v_{16}/v_{17}}{
\path[edge, dotted] (\x) -- (\y);}
\node[ver] () at (1,-1.5){$\Gamma'$};
\end{scope}

\end{tikzpicture}
\caption{A `simple cut-and-glue move' with the pairs $(z_1,z_2)$ and $(w_1,w_2)=(v_{2s-1},v_{2t})$.}\label{fig:bipartite1}
\end{figure}

Let $\Gamma$ be a contracted bipartite  3-regular colored graph with $4q+2$ vertices. Since $\Gamma$ is contracted, $\Gamma_{\{0,1\}}$ is a Hamiltonian cycle, i.e., a $(4q+2)$-cycle. Let $v_1, v_2, \dots, v_{4q+2}$ be the vertices of $\Gamma$ such that, for $1\leq i \leq 2q+1$, $v_{2i-1}v_{2i}$ are the edges of color $0$ and $v_{2i}v_{2i+1}$ are the edges of color $1$ (addition is modulo $4q+2$) in the Hamiltonian cycle. Let us denote the vertices $v_{2i-1}$ of $\Gamma$ by black dots `$\bullet$'  and the vertices $v_{2i}$ by white dots `$\circ$' for $1\leq i \leq 2q+1$. Since $\Gamma$ is bipartite, end points of each edge of color 2 are of different types. Without loss, let $v_1v_{2r}$ be an edge of color 2. If all edges with one end point lies in $\{v_2,v_3,\dots,v_{2r-1}\}$ have both end points lie in that set then $\Gamma_{\{0,2\}}$ will not be a Hamiltonian cycle. Thus, there must be an edge of color 2 between a vertex in $\{v_2,v_3,\dots,v_{2r-1}\}$ and a vertex in $\{v_{2r+1},v_{2r+2},\dots,v_{4q+2}\}$. Note that the
existence of such an edge ensures the existence of at least two such edges. Thus,
there must be an edge of color 2, whose one end point is a black dot `$\bullet$' vertex and lies in $\{v_3,\dots,v_{2r-1}\}$, and
the other end point is a white dot `$\circ$' vertex and lies in $\{v_{2r+1},v_{2r+2},\dots,v_{4q+2}\}$. Let $v_{2s-1} \in \{v_3,\dots,v_{2r-1}\}$ and $v_{2t} \in \{v_{2r+1},v_{2r+2},\dots,v_{4q+2}\}$ be such vertices.

If we choose $w_1=v_{2s-1}$ and $w_2=v_{2t}$ then, by applying a `simple cut-and-glue move' with the pairs $(z_1,z_2)$ and $(w_1,w_2)$ as in Definition \ref{def:cg}, we get a contracted 3-regular colored graph $\Gamma'$ from $\Gamma$ as in Figure \ref{fig:bipartite1}.

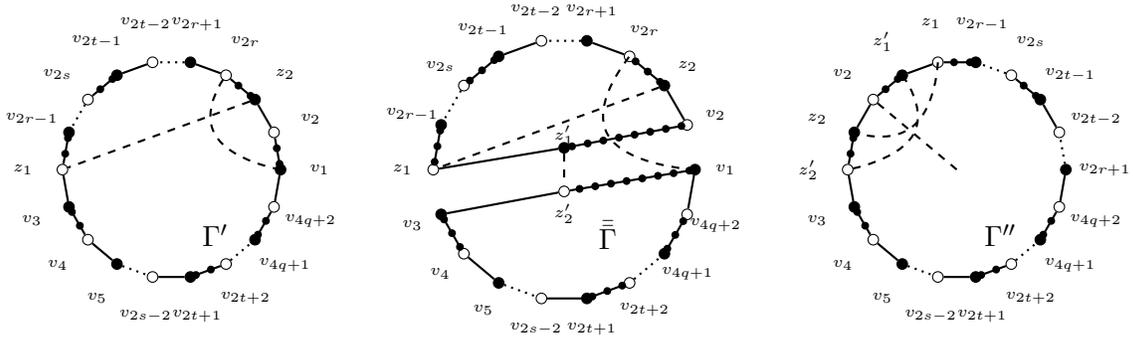
\begin{figure}[ht]
\tikzstyle{ver}=[]
\tikzstyle{vert}=[circle, draw, fill=black!100, inner sep=0pt, minimum width=4pt]
\tikzstyle{vertex}=[circle, draw, fill=black!00, inner sep=0pt, minimum width=4pt]
\tikzstyle{edge} = [draw,thick,-]
\centering
\begin{tikzpicture}[scale=0.58]
\begin{scope}[shift={(0,0)}]
\foreach \x/\y in {20/v_{2},60/v_{4},100/v_{6},140/v_{8},180/v_{10},220/v_{12},260/v_{14},300/v_{16},340/v_{18}}{
    \node[vertex] (\y) at (\x:2.5){};}

\foreach \x/\y in {0/v_{1},40/v_{3},80/v_{5},120/v_{7},160/v_{9},200/v_{11},240/v_{13},280/v_{15},320/v_{17}}{
\node[vert] (\y) at (\x:2.5){};}

\foreach \x/\y in {0/v_{1},20/v_{2},40/z_2,60/v_{2r},80/v_{2r+1},100/v_{2t-2},120/v_{2t-1},140/v_{2s},160/v_{2r-1},180/z_1,200/v_{3},220/v_{4},240/v_{5},260/v_{2s-2},280/v_{2t+1},300/v_{2t+2},320/v_{4q+1},340/v_{4q+2}}{
\node[ver] () at (\x:3.4){\tiny{$\y$}};}

\foreach \x/\y in {v_{1}/v_{2},v_{2}/v_{3},v_{3}/v_{4},v_{4}/v_{5},v_{6}/v_{7},v_{7}/v_{8},v_{9}/v_{10},v_{10}/v_{11},v_{11}/v_{12},v_{12}/v_{13},v_{14}/v_{15},v_{15}/v_{16},v_{17}/v_{18},v_{18}/v_{1}}{\path[edge] (\x) -- (\y);}

\foreach \x/\y in {v_{1}/v_{2},v_{3}/v_{4},v_{7}/v_{8},v_{9}/v_{10},v_{11}/v_{12},v_{15}/v_{16},v_{17}/v_{18}}{
\draw [line width=3pt, line cap=round, dash pattern=on 0pt off 2\pgflinewidth]  (\x) -- (\y);}

\draw[edge, dashed] plot [smooth,tension=1] coordinates{(v_{1})(1,.8)(v_{4})};
\path[edge, dashed] (v_{3}) -- (v_{10});
\foreach \x/\y in {v_{5}/v_{6},v_{8}/v_{9},v_{13}/v_{14},v_{16}/v_{17}}{
\path[edge, dotted] (\x) -- (\y);}
\node[ver] () at (1,-1.5){$\Gamma'$};
\end{scope}

\begin{scope}[shift={(9,0)}]

\foreach \x/\y in {20/v_{2},60/v_{4},100/v_{6},140/v_{8},180/v_{10},220/v_{12},260/v_{14},300/v_{16},340/v_{18}}{
    \node[vertex] (\y) at (\x:3){};}

\foreach \x/\y in {0/v_{1},40/v_{3},80/v_{5},120/v_{7},160/v_{9},200/v_{11},240/v_{13},280/v_{15},320/v_{17}}{
\node[vert] (\y) at (\x:3){};}

\node[vert] (z_1) at (0,.5){};
\node[vertex] (z_2) at (0,-.5){};
\node[ver] () at (0,.8){\tiny{$z'_1$}};
\node[ver] () at (0,-1){\tiny{$z'_2$}};

\foreach \x/\y in {0/v_{1},20/v_{2},40/z_2,60/v_{2r},80/v_{2r+1},100/v_{2t-2},120/v_{2t-1},140/v_{2s},160/v_{2r-1},180/z_1,200/v_{3},220/v_{4},240/v_{5},260/v_{2s-2},280/v_{2t+1},300/v_{2t+2},320/v_{4q+1},340/v_{4q+2}}{
\node[ver] () at (\x:3.7){\tiny{$\y$}};}

\foreach \x/\y in {v_{2}/v_{3},v_{3}/v_{4},v_{4}/v_{5},v_{6}/v_{7},v_{7}/v_{8},v_{9}/v_{10},v_{11}/v_{12},v_{12}/v_{13},v_{14}/v_{15},v_{15}/v_{16},v_{17}/v_{18},v_{18}/v_{1},v_{2}/z_1,v_{10}/z_1,v_{1}/z_2,v_{11}/z_2}{\path[edge] (\x) -- (\y);}

\foreach \x/\y in {v_{3}/v_{4},v_{7}/v_{8},v_{9}/v_{10},v_{11}/v_{12},v_{15}/v_{16},v_{17}/v_{18},v_{2}/z_1,v_{1}/z_2}{
\draw [line width=3pt, line cap=round, dash pattern=on 0pt off 2\pgflinewidth]  (\x) -- (\y);}

\draw[edge, dashed] plot [smooth,tension=1] coordinates{(v_{1})(1,.8)(v_{4})};
\path[edge, dashed] (v_{3}) -- (v_{10});
\path[edge, dashed] (z_1) -- (z_2);
\foreach \x/\y in {v_{8}/v_{9},v_{5}/v_{6},v_{13}/v_{14},v_{16}/v_{17}}{
\path[edge, dotted] (\x) -- (\y);}
\node[ver] () at (1,-1.5){$\bar{\bar{\Gamma}}$};
\end{scope}

\begin{scope}[shift={(18,0)}]
\foreach \x/\y in {20/v_{2},60/v_{4},100/v_{6},140/v_{8},180/v_{10},220/v_{12},260/v_{14},300/v_{16},340/v_{18}}{
    \node[vertex] (\y) at (\x:2.5){};}

\foreach \x/\y in {0/v_{1},40/v_{3},80/v_{5},120/v_{7},160/v_{9},200/v_{11},240/v_{13},280/v_{15},320/v_{17}}{
\node[vert] (\y) at (\x:2.5){};}

\foreach \x/\y in {0/v_{2r+1},20/v_{2t-2},40/v_{2t-1},60/v_{2s},80/v_{2r-1},100/z_{1},120/z'_1,140/v_{2},160/z_{2},180/z'_2,200/v_{3},220/v_{4},240/v_{5},260/v_{2s-2},280/v_{2t+1},300/v_{2t+2},320/v_{4q+1},340/v_{4q+2}}{
\node[ver] () at (\x:3.4){\tiny{$\y$}};}

\foreach \x/\y in {v_{2}/v_{3},v_{3}/v_{4},v_{5}/v_{6},v_{6}/v_{7},v_{7}/v_{8},v_{8}/v_{9},v_{9}/v_{10},v_{10}/v_{11},v_{11}/v_{12},v_{12}/v_{13},v_{14}/v_{15},v_{15}/v_{16},v_{17}/v_{18},v_{18}/v_{1}}{\path[edge] (\x) -- (\y);}

\foreach \x/\y in {v_{3}/v_{4},v_{5}/v_{6},v_{7}/v_{8},v_{9}/v_{10},v_{11}/v_{12},v_{15}/v_{16},v_{17}/v_{18}}{
\draw [line width=3pt, line cap=round, dash pattern=on 0pt off 2\pgflinewidth]  (\x) -- (\y);}

\draw[edge, dashed] plot [smooth,tension=1] coordinates{(v_{6})(-1,1)(v_{9})};
\draw[edge, dashed] plot [smooth,tension=1] coordinates{(v_{7})(-1,.8)(v_{10})};
\path[edge, dashed] (v_{8}) -- (0,0);
\foreach \x/\y in {v_{1}/v_{2},v_{4}/v_{5},v_{13}/v_{14},v_{16}/v_{17}}{
\path[edge, dotted] (\x) -- (\y);}
\node[ver] () at (1,-1.5){$\Gamma''$};
\end{scope}
\end{tikzpicture}
\caption{A `simple cut-and-glue move' with the pairs $(z'_1,z'_2)$ and $(w_1,w_2)=(v_{2r},v_1)$.}\label{fig:bipartite2}
\end{figure}

Again, if we choose the pairs $(z'_1,z'_2)$ and $(v_{2r},v_1)$ then, by applying a `simple cut-and-glue move' with the pairs $(z'_1,z'_2)$ and $(w_1,w_2)=(v_{2r},v_1)$ as in Definition \ref{def:cg}, we get a contracted 3-regular colored graph $\Gamma''$ from $\Gamma'$ as in Figure \ref{fig:bipartite2}.

Now, it is easy to see that $\Gamma''$ can be written as a connected sum of two graphs $G$ and $H$. Thus $\Gamma''=G\#_{uv} H$, where $V(H)=\{z_1, z'_1, v_{2},z_2,z'_2,v\}$ and $V(G)=V(\Gamma'')\cup \{u\}\setminus V(H)$ for some vertices $u$ and $v$ of different types. Here, by the construction, $H\cong \mathcal{T}_1$.

Since $|G|=4q-2$, by the assumption,  $G$ is $\mathcal{D}$-equivalent to $\mathcal{T}_{q-1}$, and hence $\Gamma$ is  $\mathcal{D}$-equivalent to $\mathcal{T}_{q-1}\# \mathcal{T}_1=\mathcal{T}_{q}$. The result now follows by the induction principle.
\end{proof}

\begin{lemma}\label{lemma:nonbipartite}
For $m \geq 2$, if $\Gamma$ is a non-bipartite contracted $3$-regular colored graph with $2m$ vertices  then $\Gamma$ is $\mathcal{D}$-equivalent to $\mathcal{P}_{m-1}$.
\end{lemma}

\begin{proof}
Let $\{0,1, 2\}$ be the color set. If $m=2$, then $\Gamma$ is a non-bipartite contracted 3-regular colored graph with four vertices. Thus, from Example \ref{eg:6-vertex}, $\Gamma$ is unique and is isomorphic to $\mathcal{P}_1$. Trivially, $\Gamma$ is $\mathcal{D}$-equivalent to $\mathcal{P}_{1}$.

Let us assume that the statement is true for $m=p-1$, i.e., if $\Gamma$ is a non-bipartite contracted 3-regular colored graph with $2p-2$ vertices then $\Gamma$ is $\mathcal{D}$-equivalent to $\mathcal{P}_{p-2}$. We now show that the statement is true for $m=p$.

Let $\Gamma$ be a non-bipartite contracted 3-regular colored graph with $2p$ vertices. Since $\Gamma$ is contracted, $\Gamma_{\{0,1\}}$ is a Hamiltonian cycle, i.e., a $2p$-cycle. Let $v_1, v_2, \dots, v_{2p}$ be the vertices of $\Gamma$ such that, for $1\leq i \leq p$, $v_{2i-1}v_{2i}$ are the edges of color $0$ and $v_{2i}v_{2i+1}$ are the edges of color $1$ (addition is modulo $2p$) in the Hamiltonian cycle. Let us denote the vertices $v_{2i-1}$ of $\Gamma$ by black dots `$\bullet$'  and the vertices $v_{2i}$ by white dots `$\circ$', for $1\leq i \leq p$. Since $\Gamma$ is non-bipartite, there is an edge of color 2 between two vertices of same type. Without loss, let there be an edge of color 2 between two black dots `$\bullet$' vertices, say $v_1$ and $v_{2r-1}$ respectively. Since there are odd number of vertices in the set $\{v_2,v_3,\dots,v_{2r-2}\}$, there must be an edge of color 2 between some vertices $v_s\in \{v_2,v_3,\dots,v_{2r-2}\}$ and $v_t \in \{v_{2r},v_{2r+1},\dots,v_{2p}\}$. ($v_s, v_t$ may be of same type or of different types.)

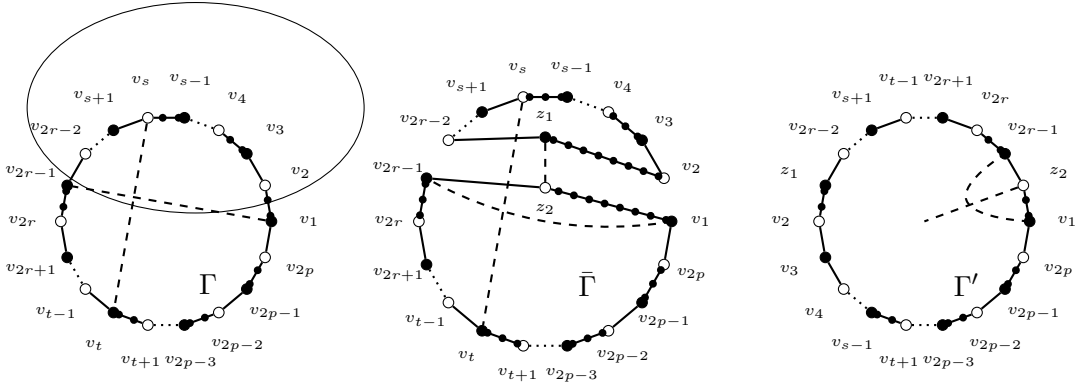
\begin{figure}[ht]
\tikzstyle{ver}=[]
\tikzstyle{vert}=[circle, draw, fill=black!100, inner sep=0pt, minimum width=4pt]
\tikzstyle{vertex}=[circle, draw, fill=black!00, inner sep=0pt, minimum width=4pt]
\tikzstyle{edge} = [draw,thick,-]
\centering
\begin{tikzpicture}[scale=0.56]
\begin{scope}[shift={(0,0)}]
\foreach \x/\y in {20/v_{2},60/v_{4},100/v_{6},140/v_{8},180/v_{10},220/v_{12},260/v_{14},300/v_{16},340/v_{18}}{
    \node[vertex] (\y) at (\x:2.5){};}

\foreach \x/\y in {0/v_{1},40/v_{3},80/v_{5},120/v_{7},160/v_{9},200/v_{11},240/v_{13},280/v_{15},320/v_{17}}{
\node[vert] (\y) at (\x:2.5){};}

\foreach \x/\y in {0/v_{1},20/v_{2},40/v_{3},60/v_{4},80/v_{s-1},100/v_{s},120/v_{s+1},140/v_{2r-2},160/v_{2r-1},180/v_{2r},200/v_{2r+1},220/v_{t-1},240/v_{t},260/v_{t+1},280/v_{2p-3},300/v_{2p-2},320/v_{2p-1},340/v_{2p}}{
\node[ver] () at (\x:3.4){\tiny{$\y$}};}

\foreach \x/\y in {v_{1}/v_{2},v_{2}/v_{3},v_{3}/v_{4},v_{5}/v_{6},v_{6}/v_{7},v_{8}/v_{9},v_{9}/v_{10},v_{10}/v_{11},v_{12}/v_{13},v_{13}/v_{14},v_{15}/v_{16},v_{16}/v_{17},v_{17}/v_{18},v_{18}/v_{1}}{\path[edge] (\x) -- (\y);}

\foreach \x/\y in {v_{1}/v_{2},v_{3}/v_{4},v_{5}/v_{6},v_{9}/v_{10},v_{13}/v_{14},v_{15}/v_{16},v_{17}/v_{18}}{
\draw [line width=3pt, line cap=round, dash pattern=on 0pt off 2\pgflinewidth]  (\x) -- (\y);}

\foreach \x/\y in {v_{1}/v_{9},v_{6}/v_{13}}{
\path[edge, dashed] (\x) -- (\y);}
\foreach \x/\y in {v_{4}/v_{5},v_{7}/v_{8},v_{11}/v_{12}, v_{14}/v_{15}}{
\path[edge, dotted] (\x) -- (\y);}
\draw (.7,2.7) ellipse (4cm and 2.5cm);
\node[ver] () at (1,-1.5){$\Gamma$};
\end{scope}

\begin{scope}[shift={(9,0)}]

    \node[vert] (z_1) at (0,2){};
    \node[vertex] (z_2) at (0,.8){};
    \node[ver] () at (0,2.5){\tiny{$z_1$}};
    \node[ver] () at (0,.3){\tiny{$z_2$}};
\foreach \x/\y in {20/v_{2},60/v_{4},100/v_{6},140/v_{8},180/v_{10},220/v_{12},260/v_{14},300/v_{16},340/v_{18}}{
    \node[vertex] (\y) at (\x:3){};}

\foreach \x/\y in {0/v_{1},40/v_{3},80/v_{5},120/v_{7},160/v_{9},200/v_{11},240/v_{13},280/v_{15},320/v_{17}}{
\node[vert] (\y) at (\x:3){};}

\foreach \x/\y in {0/v_{1},20/v_{2},40/v_{3},60/v_{4},80/v_{s-1},100/v_{s},120/v_{s+1},140/v_{2r-2},160/v_{2r-1},180/v_{2r},200/v_{2r+1},220/v_{t-1},240/v_{t},260/v_{t+1},280/v_{2p-3},300/v_{2p-2},320/v_{2p-1},340/v_{2p}}{
\node[ver] () at (\x:3.7){\tiny{$\y$}};}

\foreach \x/\y in {v_{2}/v_{3},v_{3}/v_{4},v_{5}/v_{6},v_{6}/v_{7},v_{9}/v_{10},v_{10}/v_{11},v_{12}/v_{13},v_{13}/v_{14},v_{15}/v_{16},v_{16}/v_{17},v_{17}/v_{18},v_{18}/v_{1},v_{2}/z_1,v_{8}/z_1,v_{1}/z_2,v_{9}/z_2}{\path[edge] (\x) -- (\y);}

\foreach \x/\y in {v_{3}/v_{4},v_{5}/v_{6},v_{9}/v_{10},v_{13}/v_{14},v_{15}/v_{16},v_{17}/v_{18},v_{2}/z_1,v_{1}/z_2}{
\draw [line width=3pt, line cap=round, dash pattern=on 0pt off 2\pgflinewidth]  (\x) -- (\y);}

\foreach \x/\y in {v_{6}/v_{13},z_1/z_2}{
\path[edge, dashed] (\x) -- (\y);}
\draw[edge, dashed] plot [smooth,tension=1] coordinates{(v_{1})(0,0)(v_{9})};
\foreach \x/\y in {v_{4}/v_{5},v_{7}/v_{8},v_{11}/v_{12}, v_{14}/v_{15}}{
\path[edge, dotted] (\x) -- (\y);}
\node[ver] () at (1,-1.5){$\bar\Gamma$};
\end{scope}

\begin{scope}[shift={(18,0)}]
\foreach \x/\y in {20/v_{2},60/v_{4},100/v_{6},140/v_{8},180/v_{10},220/v_{12},260/v_{14},300/v_{16},340/v_{18}}{
    \node[vertex] (\y) at (\x:2.5){};}

\foreach \x/\y in {0/v_{1},40/v_{3},80/v_{5},120/v_{7},160/v_{9},200/v_{11},240/v_{13},280/v_{15},320/v_{17}}{
\node[vert] (\y) at (\x:2.5){};}

\foreach \x/\y in {0/v_{1},20/z_2,40/v_{2r-1},60/v_{2r},80/v_{2r+1},100/v_{t-1},120/v_{s+1},140/v_{2r-2},160/z_1,180/v_2,200/v_3,220/v_{4},240/v_{s-1},260/v_{t+1},280/v_{2p-3},300/v_{2p-2},320/v_{2p-1},340/v_{2p}}{
\node[ver] () at (\x:3.4){\tiny{$\y$}};}

\foreach \x/\y in {v_{1}/v_{2},v_{2}/v_{3},v_{3}/v_{4},v_{4}/v_{5},v_{6}/v_{7},v_{8}/v_{9},v_{9}/v_{10},v_{10}/v_{11},v_{11}/v_{12},v_{13}/v_{14},v_{15}/v_{16},v_{16}/v_{17},v_{17}/v_{18},v_{18}/v_{1}}{\path[edge] (\x) -- (\y);}

\foreach \x/\y in {v_{1}/v_{2},v_{3}/v_{4},v_{9}/v_{10},v_{13}/v_{14},v_{15}/v_{16},v_{17}/v_{18}}{
\draw [line width=3pt, line cap=round, dash pattern=on 0pt off 2\pgflinewidth]  (\x) -- (\y);}

\draw[edge, dashed] plot [smooth,tension=1] coordinates{(v_{1})(1,.5)(v_{3})};
\path[edge, dashed] (v_{2}) -- (0,0);
\foreach \x/\y in {v_{5}/v_{6},v_{7}/v_{8},v_{12}/v_{13},v_{14}/v_{15}}{
\path[edge, dotted] (\x) -- (\y);}
\node[ver] () at (1,-1.5){$\Gamma'$};
\end{scope}

\end{tikzpicture}
\caption{A `simple cut-and-glue move' with the pairs $(z_1,z_2)$ and $(w_1,w_2)=(v_s,v_t)$.}\label{fig:nonbipartite}
\end{figure}

If we choose $w_1=v_s$ and $w_2=v_t$ then, by applying a `simple cut-and-glue move' with the pairs $(z_1,z_2)$ and $(w_1,w_2)$ as in Definition \ref{def:cg}, we get a contracted 3-regular colored graph $\Gamma'$ from $\Gamma$ as in Figure \ref{fig:nonbipartite}.

Now, it is easy to see that $\Gamma'$ can be written as a connected sum of two graphs $G$ and $H$. Thus $\Gamma'=G\#_{uv} H$, where $V(H)=\{v_{1}, z_2, v_{2r-1},v\}$ and $V(G)=V(\Gamma')\cup \{u\}\setminus V(H)$ for some vertices $u$ and $v$ of different types. Here, by the construction, $H\cong \mathcal{P}_1$.

Since $H$ is non-bipartite, $G$ is either bipartite or non-bipartite. If $G$ is bipartite then, by Lemmas \ref{lemma:nograph} and \ref{lemma:bipartite}, $G$ is $\mathcal{D}$-equivalent to $\mathcal{T}_{p/2-1}$, where $p$ is even (as $|G|=2p-2$). Then $\Gamma'=G\#_{uv} H$ is $\mathcal{D}$-equivalent to $\mathcal{T}_{p/2-1} \# \mathcal{P}_{1}$, and hence by Lemma \ref{lemma:connectedsum}, $\Gamma'$ is $\mathcal{D}$-equivalent to $\mathcal{P}_{p-1}$. This implies $\Gamma$ is $\mathcal{D}$-equivalent to $\mathcal{P}_{p-1}$.

On the other hand, if $G$ is non-bipartite then, by the assumption, $G$ is $\mathcal{D}$-equivalent to $\mathcal{P}_{p-2}$ (as $|G|=2p-2$), and hence $\Gamma$ is  $\mathcal{D}$-equivalent to $\mathcal{P}_{p-2}\# \mathcal{P}_1=\mathcal{P}_{p-1}$. The result now follows by the induction principle.
\end{proof}

\begin{proof}[Proof of Theorem \ref{theorem:graph}] For $n \geq 4$, Let $G$ be a contracted 3-regular colored graph with $n$ vertices. If $n=4m$ for some $m \geq 1$ then, by Lemma \ref{lemma:nograph}, $G$ is non-bipartite. Therefore, by Lemma \ref{lemma:nonbipartite}, $G$ is $\mathcal{D}$-equivalent to $\mathcal{P}_{2m-1}$. If $n=4m+2$ for some $m \geq 1$ then, by Lemmas \ref{lemma:bipartite} and \ref{lemma:nonbipartite}, $G$ is $\mathcal{D}$-equivalent to $\mathcal{T}_{m}$ (in the case when $G$ is bipartite) or $\mathcal{P}_{2m}$ (in the case when $G$ is non-bipartite). This completes the proof.
\end{proof}

\section{Proof of Corollary \ref{corollary:classification}}
In this section, we prove Corollary \ref{corollary:classification} using the following Example.

\begin{example}[Crystallizations of known closed surfaces]\label{eg:surface}
{\rm From the simplicial cell complex $\mathcal{K(L)}$ as in Figure \ref{example:S2}, it is clear that  $\mathcal{L}$ is a crystallization of $\mathbb{S}^2$.

\begin{figure}[ht]
\tikzstyle{ver}=[]
\tikzstyle{vert}=[circle, draw, fill=black!100, inner sep=0pt, minimum width=4pt]
\tikzstyle{vertex}=[circle, draw, fill=black!00, inner sep=0pt, minimum width=4pt]
\tikzstyle{edge} = [draw,thick,-]
\centering

\begin{tikzpicture}[scale=.45]

\begin{scope}[shift={(-2,0)}]
\node[vert] (v1) at (-5,0){};
\node[vertex] (v2) at (-1,0){};
\node[ver] () at (-5.5,0){\tiny{$v_{1}$}};
\node[ver] () at (-0.5,0){\tiny{$v_{2}$}};

\draw[edge] plot [smooth,tension=1] coordinates{(v1)(-3,1)(v2)};
\draw[line width=3pt, line cap=round, dash pattern=on 0pt off 2\pgflinewidth] plot [smooth,tension=1] coordinates{(v1)(-3,1)(v2)};
\path[edge] (v1) -- (v2);
\draw[edge, dashed] plot [smooth,tension=1] coordinates{(v1)(-3,-1)(v2)};
\node[ver] () at (-3,-2.5){$\mathcal{L}$};
\end{scope}

 \begin{scope}[shift={(10,0)}]

\foreach \x/\y in {45/v_{1},225/v_{2}}{
\node[ver] (\y) at (\x:.8){\tiny{$\sigma(\y)$}};
}

\node[vertex] (v_{1}) at (135:2){0};
\node[vertex] (v_{3}) at (315:2){1};

\foreach \x/\y in {45/v_{2},225/v_{4}}{
    \node[vertex] (\y) at (\x:2){2};
}
\node[ver] () at (0,1.36){$\rightarrow$};
\node[ver] () at (-1.4,0){$\downarrow$};
\node[ver] () at (1.4,0){$\uparrow$};
\node[ver] () at (1.4,0.2){$\uparrow$};
\node[ver] () at (0,-1.46){$\leftarrow$};
\node[ver] () at (0.2,-1.46){$\leftarrow$};

\foreach \x/\y in {v_{1}/v_{2},v_{2}/v_{3},v_{3}/v_{4},v_{4}/v_{1},v_{1}/v_{3}}{
\path[edge] (\x) -- (\y);}

\node[ver] () at (0,-2.5){$\mathcal{K}(\mathcal{L})$};
\end{scope}

 \begin{scope}[shift={(2,-3)}]
\node[ver] (308) at (-1,4){$0$};
\node[ver] (300) at (-1,3){$1$};
\node[ver] (301) at (-1,2){$2$};
\node[ver] (309) at (2,4){};
\node[ver] (304) at (2,3){};
\node[ver] (305) at (2,2){};

\path[edge] (300) -- (304);
\path[edge] (308) -- (309);
\draw [line width=3pt, line cap=round, dash pattern=on 0pt off 2\pgflinewidth]  (308) -- (309);
\path[edge, dashed] (301) -- (305);
\end{scope}
\end{tikzpicture}
\caption{The 2-vertex crystallization  and the corresponding simplicial cell complex of $\mathbb{S}^2$.} \label{example:S2}
\end{figure}
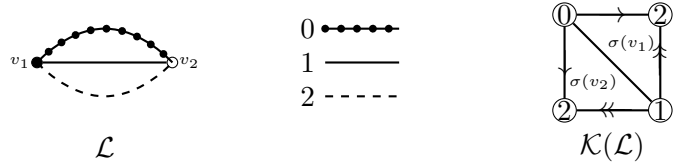

From the simplicial cell complexes $\mathcal{K}(\mathcal{P}_1)$ and $\mathcal{K}(\mathcal{T}_1)$ as in Figures \ref{example:RP2} and \ref{example:S1XS1} respectively, it is clear that, $\mathcal{P}_1$ is a crystallization of $\mathbb{RP}^2$ and $\mathcal{T}_1$ is a crystallization of $\mathbb{S}^1 \times \mathbb{S}^1$. For $m\geq 1$, $\mathcal{P}_{m}$ is a connected sum of $m$ copies of $\mathcal{P}_1$ and hence, by Proposition \ref{prop:preliminaries}, $\mathcal{P}_{m}$  is a crystallization of $\#_{m}\mathbb{RP}^2$. For $m\geq 1$,  $\mathcal{T}_{m}$ is a $(4m+2)$-vertex contracted 3-regular  colored graph which is a connected sum of $m$ copies of $\mathcal{T}_1$. Since the connected sums $G\#_{uv}\mathcal{T}_1$ and $G\#_{uw}\mathcal{T}_1$ are isomorphic for two arbitrarily chosen vertices $v,w\in V(\mathcal{T}_1)$ and for a bipartite graph $G$, it does not matter that the vertices involving in the connected sum $\mathcal{T}_{m}$ are of different types or of same type. Thus, by Proposition \ref{prop:preliminaries}, $\mathcal{T}_{m}$ is a crystallization of $\#_{m}(\mathbb{S}^1 \times \mathbb{S}^1)$.

 \begin{figure}[ht]
\tikzstyle{ver}=[]
\tikzstyle{vert}=[circle, draw, fill=black!100, inner sep=0pt, minimum width=4pt]
\tikzstyle{vertex}=[circle, draw, fill=black!00, inner sep=0pt, minimum width=4pt]
\tikzstyle{edge} = [draw,thick,-]
\centering

\begin{tikzpicture}[scale=0.5]
\begin{scope}[shift={(-6,0)}]
\foreach \x/\y in {45/v_{2},225/v_{4}}{
\node[ver] (\y) at (\x:3.5){\tiny{$\y$}};
    \node[vertex] (\y) at (\x:3){};
}

\foreach \x/\y in {135/v_{3},315/v_{1}}{
\node[ver] (\y) at (\x:3.5){\tiny{$\y$}};
    \node[vert] (\y) at (\x:3){};
}
\foreach \x/\y in {v_{1}/v_{2},v_{2}/v_{3},v_{3}/v_{4},v_{4}/v_{1}}{
\path[edge] (\x) -- (\y);}
\foreach \x/\y in {v_{1}/v_{2},v_{3}/v_{4}}{
\draw [line width=3pt, line cap=round, dash pattern=on 0pt off 2\pgflinewidth]  (\x) -- (\y);
}
\foreach \x/\y in {v_{1}/v_{3},v_{2}/v_{4}}{
\path[edge, dashed] (\x) -- (\y);}
\node[ver] () at (0,-3.5){$\mathcal{P}_1$};
\end{scope}

\begin{scope}[shift={(8,0)}]
\foreach \x/\y in {90/v_{2},270/v_{4},180/v_{3},0/v_{1}}{
\node[ver] (\y) at (\x:1.3){\tiny{$\sigma(\y)$}};
}

\foreach \x/\y in {135/v_{3},315/v_{1}}{
    \node[vertex] (\y) at (\x:3){0};
}
\foreach \x/\y in {45/v_{2},225/v_{4}}{
    \node[vertex] (\y) at (\x:3){1};
}
\node[vertex] (v) at (0,0){2};
\foreach \x/\y in {v_{1}/v_{2},v_{2}/v_{3},v_{3}/v_{4},v_{4}/v_{1},v_{1}/v,v_{2}/v,v_{3}/v,v_{4}/v}{
\path[edge] (\x) -- (\y);}

\node[ver] () at (0,2.05){$\rightarrow$};
\node[ver] () at (-2.1,-0.2){$\downarrow$};
\node[ver] () at (-2.1,0){$\downarrow$};
\node[ver] () at (2.15,0){$\uparrow$};
\node[ver] () at (2.15,0.2){$\uparrow$};
\node[ver] () at (0,-2.2){$\leftarrow$};

\node[ver] () at (0,-3.5){$\mathcal{K}(\mathcal{P}_1)$};
\end{scope}

 \begin{scope}[shift={(1,-3)}]
\node[ver] (308) at (-1,4){$0$};
\node[ver] (300) at (-1,3){$1$};
\node[ver] (301) at (-1,2){$2$};
\node[ver] (309) at (2,4){};
\node[ver] (304) at (2,3){};
\node[ver] (305) at (2,2){};

\path[edge] (300) -- (304);
\path[edge] (308) -- (309);
\draw [line width=3pt, line cap=round, dash pattern=on 0pt off 2\pgflinewidth]  (308) -- (309);
\path[edge, dashed] (301) -- (305);
\end{scope}
\end{tikzpicture}
\caption{The 4-vertex crystallization and the corresponding simplicial cell complex of $\mathbb{RP}^2$.} \label{example:RP2}
\end{figure}
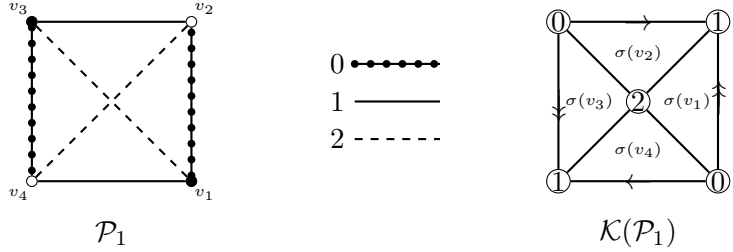

\begin{figure}[ht]
\tikzstyle{ver}=[]
\tikzstyle{vert}=[circle, draw, fill=black!100, inner sep=0pt, minimum width=4pt]
\tikzstyle{vertex}=[circle, draw, fill=black!00, inner sep=0pt, minimum width=4pt]
\tikzstyle{edge} = [draw,thick,-]
\tikzstyle{arrow} = [draw,thick,->]
\centering

\begin{tikzpicture}[scale=0.55]
\begin{scope}[shift={(-6,0)}]
\foreach \x/\y in {300/v_{2},180/v_{4},60/v_{6}}{
\node[ver] (\y) at (\x:3.5){\tiny{$\y$}};
    \node[vertex] (\y) at (\x:3){};
}
\foreach \x/\y in {240/v_{3},120/v_{5},0/v_{1}}{
\node[ver] (\y) at (\x:3.5){\tiny{$\y$}};
    \node[vert] (\y) at (\x:3){};
}
\foreach \x/\y in {v_{1}/v_{2},v_{2}/v_{3},v_{3}/v_{4},v_{4}/v_{5},v_{5}/v_{6},v_{6}/v_{1}}{
\path[edge] (\x) -- (\y);}
\foreach \x/\y in {v_{1}/v_{2},v_{3}/v_{4},v_{5}/v_{6}}{
\draw [line width=3pt, line cap=round, dash pattern=on 0pt off 2\pgflinewidth]  (\x) -- (\y);
}
\foreach \x/\y in {v_{1}/v_{4},v_{2}/v_{5},v_{3}/v_{6}}{
\path[edge, dashed] (\x) -- (\y);}
\node[ver] () at (0,-3.5){$\mathcal{T}_1$};
\end{scope}

\begin{scope}[shift={(8,0)}]
\foreach \x/\y in {330/v_{2},210/v_{4},90/v_{6}, 270/v_{3},150/v_{5},30/v_{1}}{
\node[ver] (\y) at (\x:1.6){\tiny{$\sigma(\y)$}};
}

\foreach \x/\y in {300/v_{2},180/v_{4},60/v_{6}}{
    \node[vertex] (\y) at (\x:3){0};
}
\foreach \x/\y in {240/v_{3},120/v_{5},0/v_{1}}{
    \node[vertex] (\y) at (\x:3){1};
}
\node[vertex] (v) at (0,0){2};
\foreach \x/\y in {v_{1}/v_{2},v_{2}/v_{3},v_{3}/v_{4},v_{4}/v_{5},v_{5}/v_{6},v_{6}/v_{1},v_{1}/v,v_{2}/v,v_{3}/v,v_{4}/v,v_{5}/v,v_{6}/v}{
\path[edge] (\x) -- (\y);}

\node[ver] () at (0,2.52){$\leftarrow$};
\node[ver] () at (0,-2.65){$\leftarrow$};
\node[ver] () at (0.2,2.52){$\leftarrow$};
\node[ver] () at (0.2,-2.65){$\leftarrow$};
\node[ver] () at (0.4,2.52){$\leftarrow$};
\node[ver] () at (0.4,-2.65){$\leftarrow$};
\path[arrow] (-2.3,1.2) -- (-2.2,1.3);
\path[arrow] (-2.3,-1.2) -- (-2.2,-1.3);
\path[arrow] (-2.2,-1.3) -- (-2.1,-1.4);
\path[arrow] (2.3,1.3) -- (2.4,1.2);
\path[arrow] (2.2,1.4)--(2.3,1.3);
\path[arrow] (2.3,-1.3) -- (2.4,-1.2);

\node[ver] () at (0,-3.5){$\mathcal{K}(\mathcal{T}_1)$};
\end{scope}

 \begin{scope}[shift={(1,-5)}]
\node[ver] (308) at (-1,4){$0$};
\node[ver] (300) at (-1,3){$1$};
\node[ver] (301) at (-1,2){$2$};
\node[ver] (309) at (1.5,4){};
\node[ver] (304) at (1.5,3){};
\node[ver] (305) at (1.5,2){};

\path[edge] (300) -- (304);
\path[edge] (308) -- (309);
\draw [line width=3pt, line cap=round, dash pattern=on 0pt off 2\pgflinewidth]  (308) -- (309);
\path[edge, dashed] (301) -- (305);
\end{scope}
\end{tikzpicture}
\caption{The 6-vertex crystallization and corresponding simplicial cell complex of $\mathbb{S}^1 \times \mathbb{S}^1$.} \label{example:S1XS1}
\end{figure}
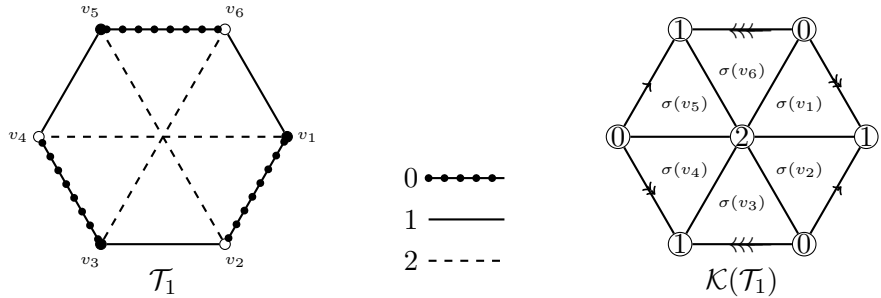

Observe that $\mathcal{P}_{m}$  (resp., $\mathcal{T}_{m}$) is non-bipartite (resp., bipartite) as the corresponding surface is non-orientable (resp., orientable). Again, by Proposition \ref{prop:preliminaries}, $\mathcal{L}\# (\#_{m}\mathcal{P}_1) \#(\#_{n}\mathcal{T}_1)$ is a crystallization of the surface $\mathbb{S}^2\# (\#_{m}\mathbb{RP}^2) \#(\#_{n}\mathbb{S}^1\times \mathbb{S}^1)$ for $m,n \geq 0$.}
\end{example}

\begin{proof}[Proof of Corollary \ref{corollary:classification}] From Proposition \ref{prop:Pezzana} we know that every closed connected surface has a crystallization, i.e., every closed connected surface can be represented by a contracted 3-regular colored graph. Let $(\Gamma,\gamma)$ be such a contracted 3-regular colored graph for the surface $S$ (i.e., $(\Gamma,\gamma)$ is a crystallization of $S$). Let $(\Gamma,\gamma)$ has $n$ vertices. If $n=2$, then $\Gamma=\mathcal{L}$ (cf. Example \ref{eg:6-vertex}), and hence $S$ is homeomorphic to $\mathbb{S}^2$ (cf. Example \ref{eg:surface}).

If $n=4m$ for some $m\geq 1$, then by Theorem \ref{theorem:graph}, $\Gamma$ is $\mathcal{D}$-equivalent to $\mathcal{P}_{2m-1}$. Since $(\Gamma,\gamma)$
 is  a crystallization of $S$ and $\mathcal{P}_{2m-1}$ is  a crystallization of $\#_{2m-1}\mathbb{RP}^2$, by the `only if part' of Corollary \ref{cor:homeomorphic}, $S$ is homeomorphic to $\#_{2m-1}\mathbb{RP}^2$.

If $n=4m+2$ for some $m\geq 1$, then by Theorem \ref{theorem:graph}, $\Gamma$ is $\mathcal{D}$-equivalent to $\mathcal{T}_{m}$ (when $\Gamma$ is bipartite) or $\mathcal{P}_{2m}$ (when $\Gamma$ is non-bipartite). We know from Proposition \ref{prop:cipartite} that $S$ is orientable if and only if $\Gamma$ is bipartite. Also, from Example \ref{eg:surface}, we know that $\mathcal{T}_{m}$ is  a crystallization of $\#_{m}(\mathbb{S}^1 \times \mathbb{S}^1)$ and $\mathcal{P}_{2m}$ is  a crystallization of $\#_{2m}\mathbb{RP}^2$. Thus, by the `only if part' of Corollary \ref{cor:homeomorphic}, $S$ is homeomorphic to $\#_{m}(\mathbb{S}^1 \times \mathbb{S}^1)$ if $S$ is orientable and is homeomorphic to $\#_{2m}\mathbb{RP}^2$ if $S$ is non-orientable. These prove the result.
\end{proof}

\begin{remark}
{\rm Because of `the existence of crystallizations of surfaces in Example \ref{eg:surface}' and the fact that `the number of vertices for crystallizations of a closed surface is unique', it is not difficult to see the following:
 (i)  Corollary \ref{corollary:classification} $\Rightarrow$ Proposition \ref{prop:Pezzana}. (ii)  Corollary \ref{corollary:classification} and the `if part' of Corollary \ref{cor:homeomorphic} $\Rightarrow$ Theorem \ref{theorem:graph},  (iii)  Theorem \ref{theorem:graph} and Corollary \ref{corollary:classification} $\Rightarrow$  Corollary \ref{cor:homeomorphic}.
In this article, we prove Theorem \ref{theorem:graph} independently. As a consequence we prove `the classification theorem of closed surfaces'. More explicitly, we prove (iv)  Theorem \ref{theorem:graph}, Proposition \ref{prop:Pezzana} and the `only if part' of  Corollary \ref{cor:homeomorphic} $\Rightarrow$ Corollary  \ref{corollary:classification}.
Thus, if we assume any three of
Theorem \ref{theorem:graph}, Corollary \ref{corollary:classification},  Proposition \ref{prop:Pezzana} and Corollary \ref{cor:homeomorphic}, then the fourth one follows.}
\end{remark}

\begin{remark}
{\rm For $d \geq 2$, let $M_1$, $M_2$ be two closed connected $d$-manifolds. If both $M_1$ and $M_2$ are orientable then there are two (possibly non-homeomorphic)
connected sums, depending on how the manifolds are oriented. From some classical results in Topology, it is known that, if $2 \leq d \leq 3$ then the connected sum of two orientable $d$-manifolds is unique up to homeomorphism. For $d=2$, this result also follows from Corollary \ref{corollary:classification}.
}
\end{remark}

\noindent {\bf Acknowledgement:} The author would like to thank Rekha Santhanam for suggesting the problem. The author thanks Basudeb Datta for the proof of Lemma \ref{lemma:nograph}, and for many useful comments and suggestions which led to the current presentation of this article. The author also thanks Bhaskar Bagchi and anonymous referees for many helpful comments. The author is supported by NBHM, India for Postdoctoral Fellowship (Award Number: 2/40(49)/2015/R$\&$D-II/11568).

\end{document}